\documentclass{amsart}

\usepackage{amsfonts}
\usepackage{amsthm}
\usepackage{amstext}
\usepackage{amsmath}
\usepackage{amscd}
\usepackage{amssymb}
\usepackage[mathscr]{eucal}
\usepackage{epsf}
\usepackage{array}
\usepackage{url}
\usepackage{accents}
\usepackage[bookmarks,bookmarksopen,bookmarksdepth=2]{hyperref}
\makeatletter
\def\@seccntformat#1{\csname the#1\endcsname.\quad}
\makeatother

\theoremstyle{plain}
\newtheorem{theorem}{Theorem}
\newtheorem{proposition}[theorem]{Proposition}

\newtheorem{lemma}[theorem]{Lemma}

\theoremstyle{definition}

\theoremstyle{remark}
\newtheorem{remark}[theorem]{Remark}

\theoremstyle{remark}

\numberwithin{equation}{section}

\newcommand{\mmod}{\operatorname{mod}}

\newcommand{\ubar}[1]{\underaccent{\bar}{#1}}

\newcommand{\ccc}{c}
\newcommand{\cccu}{\bar{\ccc}}
\newcommand{\cccl}{\ubar{\ccc}}

\newcommand{\rdet}{\operatorname{rdet}}
\newcommand{\rdetu}{\overline{\rdet}}
\newcommand{\rdetl}{\underline{\rdet}}
\newcommand{\rdetinfty}{\rdet_\infty}
\newcommand{\rdetuinfty}{\overline{\rdet}_\infty}
\newcommand{\rdetlinfty}{\underline{\rdet}_\infty}

\newcommand{\SIGMA}{\mathcal{A}}
\newcommand{\SIGMAinf}{\Sigma}
\newcommand{\LLl}{\mathcal{L}}
\newcommand{\KKk}{\mathcal{K}}
\newcommand{\NNN}{\mathbb N}
\newcommand{\ZZZ}{\mathbb Z}

\newcommand{\Orb}{\operatorname{Orb}}
\newcommand{\Per}{\operatorname{Per}}
\newcommand{\diam}{\operatorname{diam}}
\newcommand{\dist}{\operatorname{dist}}
\newcommand{\closure}[1]{\overline{#1}}

\newcommand{\eps}{\varepsilon}
\newcommand{\abs}[1]{\lvert#1\rvert}
\newcommand{\card}[1]{\##1}

\begin{document}
\bibliographystyle{abbrv}

\title{Recurrence determinism and Li-Yorke chaos for interval maps}

\author[V. \v Spitalsk\'y]{Vladim\'ir \v Spitalsk\'y}
\address{Slovanet a.s., Z\'ahradn\'icka 151, Bratislava, Slovakia}
\email{vladimir.spitalsky@slovanet.net}
\address{Department of~Mathematics, Faculty of~Natural Sciences,
    Matej Bel University, Tajovsk\'eho~40, Bansk\'a Bystrica, Slovakia}
\email{vladimir.spitalsky@umb.sk}



\begin{abstract}
Recurrence determinism, one of the fundamental characteristics of
recurrence quantification analysis, measures predictability of a trajectory of a dynamical system.
It is tightly connected with the conditional probability that, given a recurrence,
following states of the trajectory will be recurrences.

In this paper we study recurrence determinism of interval dynamical systems.
We show that recurrence determinism distinguishes three main types of $\omega$-limit sets
of zero entropy maps:
finite, solenoidal without non-separable points, 
and solenoidal with non-separable points.
As a corollary we obtain characterizations of strongly non-chaotic and Li-Yorke (non-)chaotic
interval maps via recurrence determinism. For strongly non-chaotic maps, recurrence determinism is
always equal to one. Li-Yorke non-chaotic interval maps are those for which recurrence determinism
is always positive. Finally, Li-Yorke chaos implies the existence of a Cantor set of points
with zero determinism.
\end{abstract}

\maketitle

\thispagestyle{empty}

\section{Introduction}\label{S:intro}

Recurrence plots, introduced by Eckmann et al.~\cite{eckmann1987recurrence},
provide a visual representation of trajectories of dynamical systems, which is well-suited
for data analysis.
Quantitative study of recurrence plots, called
recurrence quantification analysis \cite{zbilut1992embeddings},
has been successfully applied in many areas of science;
see \cite{webber2015recurrence} for a comprehensive overview of the subject.
One of its basic and most used characteristics is called determinism.
A slightly modified notion, which will be called here recurrence determinism
and denoted by $\rdet_m(x,\eps)$,
is tightly connected with the conditional probability that the next $m$ states of the trajectory
of a point $x$ will be $\eps$-recurrences given that the current state is an $\eps$-recurrence;
see Section~\ref{SS:determinism} for details. Thus if recurrence determinism is high,
upon encountering a recurrence we can successfully predict subsequent $m$ states of
the trajectory.

Asymptotic properties of various quantitative recurrence characteristics
were studied in
\cite{faure1998new,thiel2003analytical,zou2007analytical,faure2010recurrence,
grendar2013strong,faure2015estimating,ramdani2016recurrence},
among others.
The purpose of this paper is to show that behavior of recurrence
determinism $\rdetinfty(x,\eps)$ for small $\eps$
is able to distinguish among various types of interval dynamics.
In \cite[Theorem~B]{spitalsky2016local} it was proved that topological entropy
is the supremum of local correlation entropies. As a consequence we have that
every positive entropy interval system has a Cantor set of points $x$ whose determinism
$\rdet_m(x,\eps)$ converges to zero exponentially
fast as $m\to\infty$, and thus $\rdetinfty(x,\eps)=0$ for all sufficiently small $\eps>0$.
For zero entropy systems which are Li-Yorke chaotic, recurrence determinism
$\rdetinfty(x,\eps)$ can still be equal to zero.
On the other hand, interval maps which are not Li-Yorke chaotic have recurrence determinism
always positive. Finally, if recurrence determinism is equal to one for every
point of the interval, then the system is {strongly non-chaotic}, that is,
all $\omega$-limit sets of it are finite.
All these results are summarized in the following theorem.

\begin{theorem}\label{T:main}
  Let $f:I\to I$ be continuous. Then
  \begin{enumerate}
    \item\label{IT:T:main:1}
      $f$ is strongly non-chaotic if and only if $\rdetinfty(x,\eps)=1$
      for every $x\in I$ and every sufficiently small $\eps>0$;
    \item\label{IT:T:main:2}
      $f$ is Li-Yorke non-chaotic if and only if $\rdetlinfty(x,\eps)>0$
      for every $x\in I$ and every sufficiently small $\eps>0$;
    \item\label{IT:T:main:3}
      $f$ is Li-Yorke chaotic if and only if there is a Cantor set $C$ such that
      $\rdetinfty(x,\eps)=0$
      for every $x\in C$ and every sufficiently small $\eps>0$.
  \end{enumerate}
\end{theorem}

The next theorem states that, for any fixed point $x$,
behavior of $\rdetinfty(x,\eps)$ for small $\eps>0$
depends on the type of the $\omega$-limit set $\omega_f(x)$ of $x$.
For the corresponding definitions see Section~\ref{S:preliminaries}.

\begin{theorem}\label{T:omega}
  Let $f:I\to I$ be a continuous map with zero entropy and $x\in I$.
  Then
  \begin{enumerate}
    \item\label{IT:omega-finite}
      $\omega_f(x)$ is finite
      if and only if
      $$
        \rdetinfty(x,\eps)=1
        \qquad
        \text{for every sufficiently small }\eps>0;
      $$
    \item\label{IT:omega-nonLY}
      $\omega_f(x)$ is solenoidal and does not contain non-separable points
      if and only if
      $$
        \rdetlinfty(x,\eps)>0
        \qquad
        \text{for every }\eps>0
      $$
      and
      $$
        \liminf_{\eps\to 0}  \rdetuinfty(x,\eps) < 1;
      $$
    \item\label{IT:omega-withLY}
      $\omega_f(x)$ is solenoidal and contains non-separable points
      if and only if
      $$
        \rdetinfty(x,\eps)=0
        \qquad
        \text{for every sufficiently small }\eps>0.
      $$
  \end{enumerate}
\end{theorem}

Note that the conditions from
Theorem~\ref{T:omega}\eqref{IT:omega-nonLY} cannot be strengthened.
As was shown in \cite[Theorem~4.12]{majerova2016correlation},
determinism of a Li-Yorke non-chaotic map can be strictly smaller than $1$ for every small $\eps$,
even with $\limsup_{\eps\to 0}  \rdetuinfty(x,\eps) < 1$.
On the other hand, our final result asserts that,
for a Li-Yorke non-chaotic map with an infinite $\omega$-limit set,
there are no uniform boundaries
for recurrence determinism, that is, we cannot find $\alpha$ and $\beta$ such that
$0<\alpha\le  \rdetlinfty(x,\eps) \le  \rdetuinfty(x,\eps) \le \beta < 1$
for every sufficiently small $\eps>0$.

\begin{theorem}\label{T:example}
  There is a Li-Yorke non-chaotic map $f:I\to I$ of type $2^\infty$
  with unique infinite minimal set $C$,
  and sequences $(\eps_k)_k$, $(\eps_k')_k$ decreasing to zero
  such that, for every $x\in C$,
  $\rdetinfty(x,\eps_k) = 1$  for every $k$ and
  $\lim_k\rdetuinfty(x,\eps_k') = 0$.
  Consequently,
  $$
    \limsup_{\eps\to 0}\rdetlinfty(x,\eps) = 1
    \quad\text{and}\quad
    \liminf_{\eps\to 0}\rdetuinfty(x,\eps) = 0
    \qquad
    \text{for every } x\in C.
  $$
\end{theorem}

\medskip
In the following two sections we recall necessary
notions and facts and we prove some preliminary lemmas.
Theorems~\ref{T:main}--\ref{T:example}
are then proved in Sections~\ref{S:proof-theorem} and
\ref{S:construction}.

\section{Preliminaries}\label{S:preliminaries}
The sets of all, positive, non-negative integers are denoted by $\ZZZ$, $\NNN$, $\NNN_0$, respectively.
Let $I$ denote the unit interval $[0,1]$ equipped with the Euclidean metric $\varrho$,
and $\lambda$ denote the Lebesgue measure on $I$.
If $J,J'$ are intervals, by $J<J'$ ($J\le J'$) we mean that $y<y'$ ($y\le y'$)
for every $y\in J$ and $y'\in J'$; analogously we define $a<J$ and $a\le J$ for any real number
$a$.
If $A$ is a subset of a topological space, its boundary is denoted by $\partial A$.
The cardinality of a finite set $A$ is denoted by $\card A$.

If no confusion can arise, a set $\{m,m+1,\dots,n\}$ of consecutive integers
is denoted by $[m,n]$ or by $[m,n+1)$.
If $\alpha=\alpha_0\alpha_1\dots$ is a (finite or infinite) sequence and $0\le k<l$
are integers, by $\alpha_{[k,l)}$ or $\alpha[k,l)$ we mean $\alpha_k\alpha_{k+1}\dots\alpha_{l-1}$.

\subsection{Dynamical systems}\label{SS:dynamical-systems}

A \emph{dynamical system} is a pair $(X,f)$, where $X$ is a compact metric space
with a metric $\varrho$,
and $f:X\to X$ is a continuous map. A nonempty subset $Y$ of $X$ is called \emph{$p$-periodic}
(for some $p\in\NNN$) if $f^i(Y)$, $0\le i<p$, are pairwise disjoint and $f^p(Y)=Y$.
Note that if $Y$ is $p$-periodic and $Y'\subseteq Y$ is $p'$-periodic then
$p'$ is a multiple of $p$.

A point $x\in X$ is called \emph{$p$-periodic} or just \emph{periodic} if $\{x\}$ is $p$-periodic.
It is called \emph{non-wandering} if for every neighborhood $U$ of $x$ there is $n>0$ such that
$f^n(U)\cap U\ne\emptyset$.
The sets of all periodic and all non-wandering points of $(X,f)$ are denoted by $\Per(f)$ and
$\Omega(f)$, respectively.
The \emph{orbit} of $x$ is the set $\Orb_f(x)=\{f^n(x):\ n\in\NNN_0\}$.
The \emph{$\omega$-limit set} of $x$,
that is, the set of all limit points of the trajectory $(f^n(x))_{n\in\NNN_0}$,
is denoted by $\omega_f(x)$.

Let $f:X\to X$ be a dynamical system. A pair $(y,z)$ of points from $X$ is called a
\emph{scrambled pair} if
\begin{equation*}
  \liminf_{n\to\infty} \varrho(f^n(y), f^n(z))=0
  \qquad\text{and}\qquad
  \limsup_{n\to\infty} \varrho(f^n(y), f^n(z))>0.
\end{equation*}
A set $S\subseteq X$ is called \emph{scrambled} if $(y,z)$ is a scrambled pair for every $y\ne z$ from $S$.
The dynamical system $(X,f)$ is \emph{Li-Yorke chaotic} if there exists an uncountable scrambled
set $S\subseteq X$.

Let $f:I\to I$ be continuous. Following \cite[Definition~2.1]{smital1986chaotic},
we say that points $y,z\in I$ are \emph{separable} if there are disjoint
periodic intervals $J_y,J_z$ such that $y\in J_y$ and $z\in J_z$. If $y,z$ are distinct
and not separable, we say that they are \emph{non-separable}.
By \cite[Theorem~2.2]{smital1986chaotic} (see also \cite[Theorem~5.21]{ruette2017chaos}),
a zero entropy interval map is Li-Yorke chaotic if and only if
there exists an infinite $\omega$-limit set containing two non-separable points.

\subsection{Solenoidal $\omega$-limit sets}\label{SS:solenoidal-omega-limit}
Let $f:I\to I$ be a continuous map and $x\in I$ be such that $\omega_f(x)$
is \emph{solenoidal}. That is (see \cite[p.~4]{blokh1995spectral}), there are
a sequence of integers $2\le p_0<p_1 <\dots$ and
 a sequence of non-degenerate closed intervals $J_0\supseteq J_1\supseteq \dots$ such
that every $J_k$ is $p_k$-periodic and $\omega_f(x)\subseteq Q$, where
\begin{equation}\label{EQ:def-Q}
 Q=\bigcap_{t=0}^\infty Q_t,\qquad
 Q_t= \bigsqcup_{i=0}^{p_t-1} f^i(J_t).
\end{equation}
Put $q_0=p_0$ and $q_t=p_t/p_{t-1}$ for $t\ge 1$.
Define
\begin{equation*}
  \SIGMAinf = \prod_{i=0}^\infty \{0,\dots,q_i-1\},
  \qquad
  \SIGMA^t = \prod_{i=0}^{t-1} \{0,\dots,q_i-1\}
  \quad (t\ge 1);
\end{equation*}
every element $a=a_0a_1\dots a_{t-1}$
of $\SIGMA^t$ is called a \emph{word} and the \emph{length} of it
is $\abs{a}=t$.
Define also $\SIGMA^0=\{o\}$ (a singleton set containing the empty word $o$)
and $\SIGMA^*=\bigsqcup_{t\ge 0} \SIGMA^t$.
Let $\pi_t:\SIGMAinf\to\SIGMA^t$ ($t\ge 0$)
be the natural projection onto the first $t$ coordinates.
For $a\in\SIGMA^t$ denote by $[a]$ the set of all sequences $\alpha\in\SIGMAinf$
and all words $b\in\SIGMA^*$
\emph{starting with $a$} (i.e.~$\pi_t(\alpha)=\pi_t(b)=a$).

On $\SIGMAinf$ and on every $\SIGMA^t$ define
addition in a natural way with carry from left to right; the sets
$\SIGMAinf$ and $\SIGMA^t$ equipped with this operation are abelian groups.
Identify $10^\infty\in\SIGMAinf$ and every $10^{t-1}\in\SIGMA^t$ ($t\ge 1$) with integer $1$,
and inductively define
$\alpha+n$, $a+n$ for $\alpha\in\SIGMAinf$, $a\in\SIGMA^t$, and $n\in\ZZZ$.

For $t\ge 0$ write
$$
 Q_t = \bigsqcup_{a\in\SIGMA^t} K_a,
 \qquad
 \text{where } K_{0^t+i} =  f^i(J_t)
 \text{ for every } i\in[0,p_t).
$$
Notice that every $K_a$ ($a\in\SIGMA^t)$ is a non-degenerate closed
$p_t$-periodic interval $[y_a,z_a]$,
and $K_{b}\subseteq K_{a}$ for every $b\in[a]$.
We can also write
$$
  Q = \bigsqcup_{\alpha\in\SIGMAinf} K_\alpha,
  \qquad\text{where }
  K_\alpha = \bigcap_{t=0}^\infty K_{\pi_t(\alpha)}
  \text{ for every } \alpha\in\SIGMAinf.
$$
Here, every $K_\alpha$ is either a singleton $\{y_\alpha\}$ or a
non-degenerate closed interval $[y_\alpha,z_\alpha]$.

The following is a direct consequence of \cite[Theorem~3.1]{blokh1995spectral}.
\begin{lemma}\label{L:solenoidal-blokh}
  Let $f:I\to I$ be continuous and $x\in I$.
  Let $\omega_f(x)$ be solenoidal and $Q$ be as defined above.
  Let $C\subseteq Q\cap\closure{\Per}(f)$ be the (Cantor) set of all limit points
  of $Q\cap\Omega(f)$. Then the following assertions are true:
  \begin{enumerate}
    \item \label{IT:solenoidal-blokh-omegay}
      if $y\in Q$ then $\omega_f(y)=C$;
    \item \label{IT:solenoidal-blokh}
      if $y\in I$ and $\omega_f(y)\cap Q\ne\emptyset$ then
      $C\subseteq \omega_f(y)\subseteq Q\cap\Omega(f)$; moreover, for every $\alpha\in\SIGMAinf$,
      \begin{equation*}
        \emptyset
        \ne C\cap K_\alpha
        \subseteq \omega_f(y)\cap K_\alpha
        \subseteq \Omega(f)\cap K_\alpha
        \subseteq \partial K_\alpha.
      \end{equation*}
  \end{enumerate}
\end{lemma}

Since every $Q_t$ is $f$-invariant and has finite boundary, we immediately have
the following lemma.

\begin{lemma}\label{L:solenoidal-orbit-intersects-Qk}
  Let $\omega_f(x)$ be solenoidal and $Q=\bigcap Q_t$ be as defined above. Then
  for every $t$ there is $n_0$ such that $f^n(x)\in Q_t$ for every $n\ge n_0$.
\end{lemma}

The next lemma follows from \cite[Theorems~5.4 and 4.1(d)]{blokh1995spectral} and
\cite{smital1986chaotic};
see also \cite[Proposition~5.24]{ruette2017chaos}.

\begin{lemma}\label{L:omega-limits-for-zero-entropy}
  Let $f:I\to I$ have zero entropy and $x\in I$. Then  one of the two cases happens:
  \begin{enumerate}
    \item\label{IT:L:omega-finite} $\omega_f(x)$ is finite (and hence a periodic orbit);
    \item\label{IT:L:omega-solenoid} $\omega_f(x)$ is solenoidal.
  \end{enumerate}
   If (\ref{IT:L:omega-solenoid}) is true then $Q$ is $2$-adic, that is,
   one can choose $p_t=2^t$ for every $t$.
\end{lemma}

\begin{lemma}\label{L:nonseparable}
  Let $f:I\to I$ have zero entropy and $x\in I$ be such that $\omega_f(x)$ is solenoidal.
  Then, for every $y\ne z$ from $\omega_f(x)$,
  $y$ and $z$ are non-separable if and only if
  there is $\alpha\in\SIGMAinf$
  such that $K_\alpha$ is non-degenerate and $y,z$ are the endpoints of $K_\alpha$.
\end{lemma}
\begin{proof}
  The lemma directly follows from \cite[Lemma~5.26]{ruette2017chaos},
  which asserts that distinct $y,z\in \omega_f(x)$ are non-separable
  if and only if for every $t\ge 1$ there is $a^{(t)}\in\SIGMA^t$ such that
  $y,z\in K_{a^{(t)}}$.
\end{proof}

\section{Correlation sum and recurrence determinism}\label{S:determinism}

\subsection{Correlation sum}
Let $(X,\varrho)$ be a compact metric space and $f:X\to X$ be a continuous map.
For $x\in X$, $\eps>0$, and $n\in\NNN$, the \emph{correlation sum} is defined by
$$
  C_\varrho(x,n,\eps)=\frac{1}{n^2} \cdot
    \card\{(i,j):\ 0\le i,j < n,\ \varrho(f^i(x),f^j(x))\le \eps\}.
$$
The \emph{lower} and \emph{upper asymptotic correlation sums} are
$$
 \cccl_\varrho(x,\eps) = \liminf_{n\to\infty} C_\varrho(x,n,\eps),
 \qquad
 \cccu_\varrho(x,\eps) = \limsup_{n\to\infty} C_\varrho(x,n,\eps).
$$
Note that (asymptotic) correlation sums are numbers from the unit interval $[0,1]$,
and are equal to $1$ for $\eps \ge\diam{X}$. Note also that (asymptotic) correlation sums
are non-decreasing functions of $\eps$.

Let $m\ge 1$ be an integer.
\emph{Bowen's metric} $\varrho_m=\varrho_m^f$ is given by
$$
  \varrho_m(x,y) = \max\{\varrho(f^i(x),f^i(y)):\ 0\le i < m\}
  \qquad
  \text{for every } x,y\in X.
$$
It is a metric on $X$ compatible with the topology of $X$, thus we may define (asymptotic)
correlation sum with respect to this metric. For abbreviation, we write
$C_m(x,n,\eps)$, $\cccl_m(x,\eps)$, and $\cccu_m(x,\eps)$ instead of
$C_{\varrho_m}(x,n,\eps)$, $\cccl_{\varrho_m}(x,\eps)$, and $\cccu_{\varrho_m}(x,\eps)$,
respectively.

Since $\varrho_m\le \varrho_{m+1}$ for every $m$, we have that
$C_m(x,n,\eps)$, $\cccl_m(x,\eps)$, and $\cccu_m(x,\eps)$
are non-increasing functions of $m$.
Denote the corresponding limits, as $m$ approaches infinity, by
$C_\infty(x,n,\eps)$, $\cccl_\infty(x,\eps)$, and $\cccu_\infty(x,\eps)$, respectively.
Note that $C_\infty(x,n,\eps)=C_{\varrho_\infty}(x,n,\eps)$, where
$\varrho_\infty$ is a metric given by
$\varrho_\infty(x,y) = \sup\{\varrho(f^i(x),f^i(y)):\ 0\le i < \infty\}$. In general,
however, the metric $\varrho_\infty$ need not be compatible with the topology of $X$;
e.g.~for an expansive system $(X,f)$ the metric $\varrho$ is always discrete.

The following is Lemma~8 from \cite{spitalsky2016local}.
\begin{lemma}\label{L:corr-sum-lower-bound}
  Let $X$ be a compact metric space and $\eps>0$. Then there is $\eta\in(0,1)$ such that,
  for every continuous map $f:X\to X$ and every $x\in X$,
  $$
    \cccu_m(x,\eps)\ge \cccl_m(x,\eps)\ge\eta^m
    \qquad
    \text{for every } m \in\NNN.
  $$
\end{lemma}

\subsection{Recurrence determinism}\label{SS:determinism}
For $x\in X$, $\eps>0$, $m\in\NNN\cup\{\infty\}$, and $n\ge 1$, define the
\emph{recurrence $m$-determinism} by
$$
  \rdet_m(x,n,\eps) = \frac{C_m(x,n,\eps)}{C_1(x,n,\eps)} \,,
$$
and the \emph{upper} and \emph{lower asymptotic recurrence $m$-determinism} by
$$
  \rdetl_m(x,\eps) = \liminf_{n\to\infty} \rdet_m(x,n,\eps),
  \qquad
  \rdetu_m(x,\eps) = \limsup_{n\to\infty} \rdet_m(x,n,\eps).
$$
If $\rdetl_m(x,\eps)=\rdetu_m(x,\eps)$ we say that the
\emph{recurrence $m$-determinism exist} and we denote the
common value by $\rdet_m(x,\eps)$; analogously for $\ccc_m(x,\eps)$.
\begin{remark}[RQA-determinism]
\label{R:RQA-det}
  Note that, for finite $m$,
  the definition of recurrence
  determinism slightly differs from that used in recurrence quantification analysis.
  Since here we deal with embedding dimension $1$,
  RQA-determinism is
  $\operatorname{DET}_m(x,n,\eps)=m\cdot\rdet_m(x,n,\eps) - (m-1)\cdot\rdet_{m+1}(x,n,\eps)$
  by \cite[Proposition~1]{grendar2013strong}, see also \cite[Theorem~2]{schultz2015approximation}.
  However, for $m=\infty$ it holds that $\operatorname{DET}_\infty(x,n,\eps)=\rdetinfty(x,n,\eps)$
  by \cite[Lemma~2.1]{majerova2016correlation}.
\end{remark}

\begin{remark}[Recurrence determinism as a conditional probability]
\label{R:conditional-probability}
  Asymptotic recurrence determinism is (for typical $x$ and $\eps$) equal to
  the conditional probability that the following $m$ states of the trajectory will be recurrences
  given that the current state is a recurrence; here, by a recurrence we mean that the
  distance of a state from some previous one is smaller than or equal to the precision $\eps$.
  To be more precise, take an ergodic measure $\mu$ of the system $(X,f)$.
  Then, by \cite{pesin1993rigorous}, for $\mu$-a.e.~$x\in X$ and for all but countably many $\eps>0$,
  asymptotic correlation sum exists and is equal to the \emph{correlation integral} of $\mu$
  $$
    \ccc_m(\mu,\eps)=\mu\times\mu\{(y,z)\in X\times X:\ \varrho_m(y,z)\le\eps\}.
  $$
  Thus if $Y,Z$ are independent $X$-valued random variables with distribution $\mu$, then
  asymptotic correlation sum $\ccc_m(x,\eps)$ is (typically) equal to the probability
  that $f^i(Y)$ and $f^i(Z)$ are $\eps$-close for every $i\in[0,m)$.
  Consequently, asymptotic recurrence determinism $\rdet_m(x,\eps)$
  is (typically) equal to the conditional probability
  that $f^i(Y)$ and $f^i(Z)$ are $\eps$-close for every $i\in[0,m)$,
  given that $Y$ and $Z$ are $\eps$-close.
  For more details see \cite{grendar2013strong}.
\end{remark}

Asymptotic correlation sum and asymptotic recurrence determinism do not depend on the beginning of
the trajectory, as is stated in the next lemma.
\begin{lemma}\label{L:c(fx)=c(x)}
  Let $(X,f)$ be a dynamical system, $x\in X$, $h\in\NNN$, $m\in\NNN\cup\{\infty\}$, and $\eps>0$.
  Then
  \begin{equation*}
    \cccu_m(f^h(x),\eps) = \cccu_m(x,\eps),
    \qquad
    \cccl_m(f^h(x),\eps) = \cccl_m(x,\eps),
  \end{equation*}
  and
  \begin{equation*}
    \rdetu_m(f^h(x),\eps) = \rdetu_m(x,\eps),
    \qquad
    \rdetl_m(f^h(x),\eps) = \rdetl_m(x,\eps).
  \end{equation*}
\end{lemma}
\begin{proof}
  It suffices to apply the following inequalities for correlation
  sums, valid for every integer $n$:
  \begin{equation*}
    \left( \frac{n+h}{n} \right)^2   C_m(x,n+h,\eps) - \frac{2hn+h^2}{n^2}
    \ \le\
    C_m(f^h(x),n,\eps)
    \ \le\
    \left( \frac{n+h}{n} \right)^2   C_m(x,n+h,\eps)  \,.
  \end{equation*}
\end{proof}

The following result easily follows from Lemma~\ref{L:corr-sum-lower-bound}
and from basic properties of $\liminf$ and $\limsup$.
\begin{lemma}\label{L:rdet-and-c}
  Let $(X,f)$ be a dynamical system, $x\in X$, $m\in\NNN\cup\{\infty\}$, and $\eps>0$.
  Then $\cccl_1(x,\eps)>0$ and
  $$
    \frac{\cccl_m(x,\eps)}{\cccu_1(x,\eps)}
    \ \le\ 
    \rdetl_m(x,\eps)
    \ \le\
    \rdetu_m(x,\eps)
    \ \le\
    \frac{\cccu_m(x,\eps)}{\cccl_1(x,\eps)} \,.
  $$
\end{lemma}

\section{Proofs of Theorems~\ref{T:main} and \ref{T:omega}}\label{S:proof-theorem}
Here we give proofs of Theorems~\ref{T:main} and \ref{T:omega}.
In Section~\ref{SS:finite-omega} we deal with the simplest case
of finite $\omega$-limit sets.
Some simple lemmas regarding solenoidal $\omega$-limit sets are given
in Section~\ref{SS:solenoidal-basic}.
The case when an $\omega$-limit set contains non-separable points is
described in Section~\ref{SS:solenoidal-with-nonseparable}.
Finally, in Sections~\ref{SS:solenoidal-without-nonsepar1} and
\ref{SS:solenoidal-without-nonsepar2} we deal with the remaining case
of solenoidal $\omega$-limit sets without non-separable points.

\subsection{Finite $\omega$-limit sets}\label{SS:finite-omega}

\begin{proposition}\label{P:finite-omega}
Let $(X,f)$ be a dynamical system and $x\in X$ have
$\omega_f(x)$ of finite cardinality $p$.
Then, for every sufficiently small $\eps>0$ and every $m\in\NNN\cup\{\infty\}$,
\begin{equation*}
 \ccc_m(x,\eps)=\frac 1p
 \qquad
 \text{and}
 \qquad
 \rdet_m(x,\eps)=1.
\end{equation*}
\end{proposition}
\begin{proof}
We may assume that $m$ is finite.
Write $x_i=f^i(x)$ for $i\ge 0$ and
$\omega_f(x)=\{y_0,y_1,\dots,y_{p-1}\}$, where
$f^k(y_0)=y_{k\mmod p}$ for every $k\ge 0$.
Take arbitrary $\eps>0$ such that $\varrho(y_k,y_l)\ge 2\eps$ for every $k\ne l$.
Since $x$ is attracted by the $p$-cycle $(y_0,y_1,\dots,y_{p-1})$,
there is $i_0$ such that
$\varrho(x_{i_0+i},y_{i \mmod p})<\eps/2$ for every $i\ge 0$.
By Lemma~\ref{L:c(fx)=c(x)} we may assume that $i_0=0$.
Hence, for $i,j\ge 0$
with $i\equiv j (\mmod p)$ we have
$$
 \varrho(x_{i},x_{j})
 \le
 \varrho(x_{i},y_{i\mmod p}) + \varrho(y_{i\mmod p}, y_{j\mmod p}) + \varrho(y_{j\mmod p}, x_{j})
 <\eps,
$$
and for $i,j\ge 0$ with $i\not\equiv j (\mmod p)$ we have
$$
 \varrho(x_{i},x_{j})
 \ge
 \varrho(y_{i\mmod p}, y_{j\mmod p}) - \varrho(x_{i},y_{i\mmod p}) - \varrho(y_{j\mmod p}, x_{j})
 >2\eps - \eps = \eps.
$$
That is, for any $i,j\ge 0$, $\varrho_m(x_i,x_j)\le\eps$ if and only $i\equiv j (\mmod p)$.
Now the assertion easily follows.
\end{proof}

\subsection{Solenoidal $\omega$-limit sets --- basic facts}\label{SS:solenoidal-basic}
Till the end of Section~\ref{S:proof-theorem} we will use the notation from
Section~\ref{SS:solenoidal-omega-limit}. Fix a map $f:I\to I$ and a point $x\in I$,
and put $x_i=f^i(x)$ for every $i\ge 0$.
Take $m\in\NNN\cup\{\infty\}$ and $t\ge 0$. For every $a,b\in\SIGMA^t$ put
\begin{equation*}
  \dist_m(K_a,K_b) = \max_{i\in[0,m)} \dist(K_{a+i},K_{b+i}),
  \qquad
  \diam_m(K_a,K_b) = \max_{i\in[0,m)} \diam(K_{a+i}\cup K_{b+i}).
\end{equation*}
Note that, due to periodicity of intervals $K_a$, $\dist_m=\dist_{p_t}$ and $\diam_m=\diam_{p_t}$
for every $m\ge p_t$.
For $\eps>0$ define
\begin{eqnarray*}
  N_m(x,t,\eps) &=& \card\{(a,b)\in \SIGMA^t\times\SIGMA^t\colon \dist_m(K_a,K_b) < \eps\},
  \\
  N_m^\circ(x,t,\eps) &=& \card\{(a,b)\in \SIGMA^t\times\SIGMA^t\colon \diam_m(K_a, K_b) \le \eps\}.
\end{eqnarray*}

\begin{lemma}\label{L:corr-integ-and-intervals}
  Let $f:I\to I$ be continuous and $x\in I$ be such that $\omega_f(x)$ is solenoidal.
  Then, for every $m\in\NNN\cup\{\infty\}$, $t\ge 0$, and $\eps>0$,
  \begin{enumerate}
    \item\label{IT:L:corr-integ-and-intervals:1}
      $N_m(x,t,\eps)\ge N_{m+1}(x,t,\eps)$ and $N_m^{\circ}(x,t,\eps)\ge N_{m+1}^{\circ}(x,t,\eps)$;
    \item\label{IT:L:corr-integ-and-intervals:2}
      $N_m(x,t,\eps)=N_{p_t}(x,t,\eps)$ and $N_m^\circ(x,t,\eps)=N_{p_t}^\circ(x,t,\eps)$
      if $m\ge p_t$;
    \item\label{IT:L:corr-integ-and-intervals:3}
      $N_m^\circ(x,t,\eps) \le N_m(x,t,\eps)$;
    \item\label{IT:L:corr-integ-and-intervals:4}
      $N_m^\circ(x,t,\eps)\cdot p_t^{-2} \le \cccl_m(x,\eps)\le \cccu_m(x,\eps)
       \le N_m(x,t,\eps)\cdot p_t^{-2}$;
    \item\label{IT:L:corr-integ-and-intervals:5}
      if $N_1^\circ(x,t,\eps)>0$ then
      $$
        \frac{N_m^\circ(x,t,\eps)}{N_1(x,t,\eps)}
        \le
        \rdetl_m(x,\eps)
        \le
        \rdetu_m(x,\eps)
        \le
        \frac{N_m(x,t,\eps)}{N_1^\circ(x,t,\eps)}\,.
      $$
  \end{enumerate}
\end{lemma}
\begin{proof}
  Properties \eqref{IT:L:corr-integ-and-intervals:1}, \eqref{IT:L:corr-integ-and-intervals:2}
  are trivial and \eqref{IT:L:corr-integ-and-intervals:3} follows from the fact that every
  $K_a$ ($a\in\SIGMA^t$) is non-degenerate.
  To prove (\ref{IT:L:corr-integ-and-intervals:4}), by Lemma~\ref{L:c(fx)=c(x)} and the fact that
  $\omega_f(x)$ is infinite we may assume that
  $x\in K_{0^t}$ and that the orbit of $x$ does not intersect
  the boundary of any $K_a$ ($a\in\SIGMA^t$).
  Thus $x_i\in K_{0^t+i}$ for every $i$, and
  $$
    \dist_m(K_{0^t+i},K_{0^t+j})
    <
    \varrho_m(x_i,x_j)
    <
    \diam_m(K_{0^t+i},K_{0^t+j})
    \qquad\text{for every }i\ne j.
  $$
  Since intervals $K_a$ ($a\in\SIGMA^t$) are $p_t$-periodic, we have
  $$
    \frac{k_n^2}{(k_n+1)^2 \cdot p_t^2} \cdot N_m^\circ(x,t,\eps)
    <
    C_m(x,n,\eps)
    <
    \frac{(k_n+1)^2}{k_n^2 \cdot p_t^2} \cdot N_m(x,t,\eps)
  $$
  for every $n\ge p_t$, where $k_n=\lfloor n/p_t\rfloor$.
  A passage to the limit $n\to\infty$ gives \eqref{IT:L:corr-integ-and-intervals:4}.

  Since \eqref{IT:L:corr-integ-and-intervals:5} follows from
  \eqref{IT:L:corr-integ-and-intervals:4}
  and Lemma~\ref{L:rdet-and-c}, the proof is finished.
\end{proof}

The following proposition states that, for zero entropy interval maps,
asymptotic correlation sum $\ccc_\infty$ can distinguish points with finite $\omega$-limit set
from those with infinite one.

\begin{proposition}\label{P:finite-omega-and-cinfty}
  Let $f:I\to I$ be a continuous map with zero entropy and $x\in I$.
  Then  the following conditions are equivalent:
  \begin{enumerate}
    \item\label{IT:P:finite-omega-and-cinfty:1}
      $\omega_f(x)$ is finite;
    \item\label{IT:P:finite-omega-and-cinfty:2}
      there is $p\in\NNN$ such that $\ccc_\infty(x,\eps)=1/p$ for every sufficiently small $\eps>0$;
    \item\label{IT:P:finite-omega-and-cinfty:3}
      $\liminf_{\eps\to 0} \cccl_\infty(x,\eps)>0$;
    \item\label{IT:P:finite-omega-and-cinfty:4}
      $\limsup_{\eps\to 0} \cccu_\infty(x,\eps)>0$.
  \end{enumerate}
\end{proposition}
\begin{proof}
  The fact that \eqref{IT:P:finite-omega-and-cinfty:1} implies \eqref{IT:P:finite-omega-and-cinfty:2}
  was proved in Proposition~\ref{P:finite-omega}, and
  \eqref{IT:P:finite-omega-and-cinfty:2}$\implies$\eqref{IT:P:finite-omega-and-cinfty:3}
  $\implies$\eqref{IT:P:finite-omega-and-cinfty:4}
  is obvious.
  To prove \eqref{IT:P:finite-omega-and-cinfty:4}$\implies$\eqref{IT:P:finite-omega-and-cinfty:1}
  assume that $\omega_f(x)$ is infinite, hence solenoidal with $p_t=2^t$
  by Lemma~\ref{L:omega-limits-for-zero-entropy}.
  Take any $t\in\NNN$ and $0<\eps\le\min\{\dist(K_a,K_b):\ a,b\in\SIGMA^t,\ a\ne b\}$.
  Then $\cccu_\infty(x,\eps)\le 2^{-t}$ by
  Lemma~\ref{L:corr-integ-and-intervals}\eqref{IT:L:corr-integ-and-intervals:4}.
  Since $t$ is arbitrary, we have $\lim_\eps \cccu_\infty(x,\eps)=0$.
\end{proof}

\subsection{Solenoidal $\omega$-limit sets with non-separable points}
\label{SS:solenoidal-with-nonseparable}
The following is related to \cite[Lemma~4.1]{smital1986chaotic}.

\begin{lemma}\label{L:solenoid-noLYC-non_isolated}
  Let $f:I\to I$ be continuous and $x\in I$ be such that $\omega_f(x)$ is solenoidal.
  Let $\alpha\in \SIGMAinf$ be such that $K_\alpha=[y_\alpha,z_\alpha]$
  is non-degenerate and both endpoints $y_\alpha,z_\alpha$ belong to $\omega_f(x)$.
  Then for every $s\in\NNN$ there are $t=t^{(s)}>s$ and $b=b^{(s)},c=c^{(s)}\in \SIGMA^t$
  such that
  $$
    b_{[0,s)}=c_{[0,s)}=\alpha_{[0,s)}
    \qquad\text{and}\qquad
    K_b < K_{\alpha{[0,t)}} < K_c.
  $$
  Moreover, $t^{(s)} < t^{(s+1)}$ for every $s$.
\end{lemma}
\begin{proof}
  Put $t^{(0)}=0$. Fix an integer $s\ge 1$ and assume that $t^{(s-1)}$ has been defined.
  Since $\omega_f(x)$ is infinite, $y_\alpha$ and $z_\alpha$ are not isolated points of
  $\omega_f(x)$. Moreover, the interior $(y_\alpha,z_\alpha)$ of $K_\alpha$ is wandering
  \cite[Corollary~3.2(2)]{blokh1995spectral}.
  Thus there are increasing sequences $(n_k)_k$ and $(m_k)_k$ of integers such that
  $f^{n_k}(x) \nearrow y_\alpha$ and $f^{m_k}(x) \searrow z_\alpha$; we may assume that
  $f^{n_k}(x),f^{m_k}(x)\in Q_s$ for all sufficiently large $k$. This, together with
  $y_\alpha,z_\alpha\in K_{\alpha[0,s)}$, implies that there is $k_0$ such that
  $f^{n_k}(x),f^{m_k}(x)\in K_{\alpha[0,s)}$ for every $k\ge k_0$.

  Fix any $k\ge k_0$.
  Since $f^{n_k}(x)<y_\alpha$, there is $t_0 > s$ such that
  $f^{n_k}(x)\in Q_{t_0}\setminus K_{\alpha[0,t_0)}$. Hence there is $b'\in \SIGMA^{t_0}$
  with $f^{n_k}(x)\in K_{b'} < K_{\alpha[0,t_0)}$; moreover, by the choice of $k_0$,
  $b'_{[0,s)}=\alpha_{[0,s)}$. Analogously, there are $t_1>s$
  and $c'\in \SIGMA^{t_1}$ such that $f^{m_k}(x)\in K_{c'} > K_{\alpha[0,t_0)}$ and
  $c'_{[0,s)}=\alpha_{[0,s)}$. Now it suffices to put $t=t^{(s)}=\max\{t_0,t_1,t^{(s-1)}+1\}$ 
  and take arbitrary $b,c\in\SIGMA^t$ with $b_{[0,t_0)}=b'$ and $c_{[0,t_1)}=c'$.
\end{proof}

\begin{lemma}\label{L:solenoid-LYC-zero-det}
Let $f:I\to I$ be continuous and $x\in I$ be such that $\omega_f(x)$ is solenoidal.
Assume that $\alpha\in\SIGMAinf$ and $\eps>0$ are such that $\diam{K_\alpha}\ge\eps$ and
$\partial K_\alpha \subseteq \omega_f(x)$.
Then
$$
 \ccc_\infty(x,\eps)=0
 \qquad\text{and }\quad
 \rdetinfty(x,\eps)=0.
$$
\end{lemma}
\begin{proof}
Fix arbitrary $s\ge 1$ and $h\in \NNN \setminus p_s\NNN$.
Let $t=t^{(s)}$, $b=b^{(s)}$, and $c=c^{(s)}$ be as in Lemma~\ref{L:solenoid-noLYC-non_isolated}.
By Lemma~\ref{L:c(fx)=c(x)} we may assume that $x\in K_{0^t}$.
Let $j_0$ be such that $0^s+j_0=\alpha_{[0,s)}$. Put $u=0^s+h+j_0$; then
$u\ne \alpha_{[0,s)}$ by the choice of $h$.

Assume first that $K_u<K_{\alpha[0,s)}$.
Let $j_1\in[0,p_t)$ be such that $0^t+j_0+j_1=c$.
Since $c_{[0,s)}=\alpha_{[0,s)}$ and $0^t+j_0$ starts with $\alpha_{[0,s)}$,
$j_1$ is a multiple of $p_s$. Further,
$(0^t+h+j_0+j_1)_{[0,s)}=0^s+h+j_0 = u$.
Thus $x_{j_0+j_1+lp_t}\in K_c$ and $x_{h+j_0+j_1+lp_t}\in K_u$ for every $l\ge 0$.
Since $K_u < \min K_{\alpha[0,s)} \le K_{\alpha[0,t)} < K_c$,
we conclude that
$$
  \varrho_{\infty}(x_i,x_{h+i}) \ge
  \varrho_{p_t}(x_i,x_{h+i}) > \diam K_{\alpha[0,t)} > \diam K_\alpha \ge\eps
  \qquad\text{for every } i\ge 0.
$$

The same conclusion can be analogously obtained in the case when $K_u>K_{\alpha[0,s)}$;
one only needs to replace $c$ with $b$.
Thus in both cases we have $\varrho_{\infty}(x_i,x_{h+i}) > \eps$ for every $i\ge 0$ and every
$h\in\NNN\setminus p_s\NNN$. This immediately implies that
$\cccu_{\infty}(x,\eps)\le (1/p_s)$.
Since $s$ is arbitrary and $\lim_s p_s=\infty$,
$\ccc_\infty(x,\eps)=0$ and hence, by Lemma~\ref{L:rdet-and-c}, $\rdetinfty(x,\eps)=0$.
\end{proof}

\begin{proposition}\label{P:solenoid-LYC}
Let $f:I\to I$ be continuous and $x\in I$.
If $\omega_f(x)$ is a solenoidal $\omega$-limit set containing two non-separable points
$y$ and $z$, then
$$
 \ccc_\infty(x,\eps)=\rdetinfty(x,\eps)=0
 \qquad
 \text{for every }
 \eps\le \varrho(y,z) \,.
$$
\end{proposition}
\begin{proof}
  The assertion follows immediately from Lemmas~\ref{L:solenoid-LYC-zero-det}
  and \ref{L:nonseparable}.
\end{proof}

In the proof of Theorem~\ref{T:main} we will need the following lemma.
\begin{lemma}\label{L:nonseparable-in-C}
  Let $f:I\to I$ have zero entropy and $x\in I$ be such that $\omega_f(x)$ is solenoidal
  and contains non-separable points $y,z$. Let $C$ be the set from Lemma~\ref{L:solenoidal-blokh}.
  Then $y,z\in C$.
\end{lemma}
\begin{proof}
  By Lemma~\ref{L:nonseparable} we may assume that $y=y_\alpha$ and $z=z_\alpha$
  for some $\alpha\in\SIGMAinf$.
  Lemma~\ref{L:solenoid-noLYC-non_isolated} asserts that any
  open interval $J\ni y$ contains a (periodic) interval $K_b\not\ni y$.
  Hence $y$ is not an isolated point of $Q\cap\Omega(f)$ and so it belongs to $C$ 
  (recall that $C$ is the set of all limit points of $Q\cap\Omega(f)$). 
  Analogously we can prove that $z\in C$.
\end{proof}

\subsection{Solenoidal $\omega$-limit sets without non-separable points --- the first part}
\label{SS:solenoidal-without-nonsepar1}

\begin{proposition}\label{P:solenoid-notLYC-lower-bound}
Let $f:I\to I$ be continuous. Let $x\in I$ be such that $\omega_f(x)$ is solenoidal
and does not contain non-separable points. Then
$$
 \rdetlinfty(x,\eps)
 \ge
 \cccl_\infty(x,\eps)
 > 0
 \qquad
 \text{for every } \eps>0.
$$
\end{proposition}
\begin{proof}
Fix any $\eps>0$. Let $\{\alpha^{(1)},\dots,\alpha^{(m)}\}$ be 
the set of all $\alpha\in\SIGMAinf$ with $\diam K_\alpha \ge (\eps/2)$.
Since interiors of $K_{\alpha^{(j)}}$ are wandering \cite[Corollary~3.2(2)]{blokh1995spectral}
and $x$ is not eventually periodic, the orbit of $x$ visits $\bigcup_j K_{\alpha^{(j)}}$
at most finitely many times; by Lemma~\ref{L:c(fx)=c(x)} we may assume that
\begin{equation}\label{EQ:solenoid-notLYC-lower-bound:orbit}
  K_{\alpha^{(j)}} \cap \Orb_f(x)  =\emptyset
  \qquad
  \text{for every } j\in [1,m].
\end{equation}

Fix any $j\in [1,m]$. Since $\omega_f(x)$ contains no non-separable points,
\begin{equation}\label{EQ:solenoid-notLYC-lower-bound:singleton}
  K_{\alpha^{(j)}} \cap \omega_f(x)
  \quad
  \text{is a singleton subset of }
  \quad
  \partial K_{\alpha^{(j)}} = \{y_{\alpha^{(j)}},z_{\alpha^{(j)}}\}
\end{equation}
by Lemmas~\ref{L:solenoidal-blokh} and \ref{L:nonseparable}.
Thus there is $s_j\in\NNN$ with the following properties:
if $y_{\alpha^{(j)}}\in\omega_f(x)$ then
        \begin{equation}\label{EQ:solenoid-notLYC-lower-bound:Ka-less}
          K_a < K_{\alpha^{(j)}[0,t)}
          \qquad
          \text{for every }  t\ge s_j \text{ and }
          a\in \SIGMA^t\setminus\{\alpha^{(j)}_{[0,t)}\}
          \text{ starting with }      \alpha^{(j)}_{[0,s_j)};
        \end{equation}
and if $z_{\alpha^{(j)}}\in\omega_f(x)$ then
        \begin{equation}\label{EQ:solenoid-notLYC-lower-bound:Ka-greater}
          K_a > K_{\alpha^{(j)}[0,t)}
          \qquad
          \text{for every }  t\ge s_j \text{ and }
          a\in \SIGMA^t\setminus\{\alpha^{(j)}_{[0,t)}\}
          \text{ starting with }      \alpha^{(j)}_{[0,s_j)}.
        \end{equation}
To see this it suffices to realize that if \eqref{EQ:solenoid-notLYC-lower-bound:Ka-less}
is not true then $z_{\alpha^{(j)}}\in\omega_f(x)$,
and if \eqref{EQ:solenoid-notLYC-lower-bound:Ka-greater} is not true then
$y_{\alpha^{(j)}}\in\omega_f(x)$.

We claim that there exists $s\ge\max\{s_j:\ j\in[1,m]\}$ such that, for every $b\in\SIGMA^s$,
\begin{equation}\label{EQ:solenoid-notLYC-lower-bound:diam}
  \text{either }
  b=\alpha^{(j)}_{[0,s)}
     \text{ and }
     \lambda(K_b\setminus K_{\alpha^{(j)}}) < \eps
     \text{ for some }j,
  \qquad\text{or }
  \diam K_b < \eps.
\end{equation}
To prove it suppose that there are a sequence $(s_l')_l\nearrow\infty$ of integers
and a sequence $(b^{(l)})_l$ of words $b^{(l)}\in \SIGMA^{s_l'}$,
such that $\diam K_{b^{(l)}}\ge \eps$ and $b^{(l)}$ is not a beginning of any $\alpha^{(j)}$.
Denote by $w_l$ the middle of $K_{b^{(l)}}$; without loss of generality we may assume that
$(w_l)_l$ converges to some $w\in I$. Since $w_l \pm (\eps/2) \in K_{b^{(l)}}$, we have that
$L=[w-\eps/4,w+\eps/4]$ is a subset of $K_{b^{(l)}}$ for every sufficiently large $l$.
Hence $L\subseteq Q$ and so, due to $\diam L= \eps/2$,
$L$ is a subset of some $K_{\alpha^{(j)}}$.
But then $K_{\alpha^{(j)}}\subseteq K_{b^{(l)}}$, that is, $b^{(l)}$ is a beginning of $\alpha^{(j)}$,
a contradiction. Thus we have proved that there is arbitrarily large $s$ such that,
for every $b\in \SIGMA^s$, either $b=\alpha^{(j)}_{[0,s)}$ or $\diam K_b<\eps$.
Since $K_{\alpha^{(j)}}=\bigcap_s K_{\alpha^{(j)}[0,s)}$, we see that for $s$ large enough
we also have $\lambda(K_b\setminus K_{\alpha^{(j)}}) < \eps$. Hence
\eqref{EQ:solenoid-notLYC-lower-bound:diam} is proved.

We may assume that $x\in Q_s$.
Take any $h\in p_s\NNN$. Since $h$ is a multiple of $p_s$, for every $i$ there is
$b^{(i)}\in \SIGMA^s$ with $x_i,x_{i+h}\in K_{b^{(i)}}$. We claim that
\begin{equation}\label{EQ:solenoid-notLYC-lower-bound:dist}
  \abs{x_{i+h}-x_i}\le\eps
  \qquad\text{for every }i.
\end{equation}
To prove it, fix any $i$ and put $b=b^{(i)}$. If $b\ne \alpha^{(j)}_{[0,s)}$ for every $j$,
then $\diam K_b<\eps$ by \eqref{EQ:solenoid-notLYC-lower-bound:diam}, and
\eqref{EQ:solenoid-notLYC-lower-bound:dist} is immediate.
So assume that $b = \alpha^{(j)}_{[0,s)}$ for some $j$; put $\alpha=\alpha^{(j)}$.
By \eqref{EQ:solenoid-notLYC-lower-bound:singleton}, either $y_\alpha\in \omega_f(x)$
or $z_\alpha\in \omega_f(x)$; we will consider only the former case, the latter one being similar.
Since $x_i\not\in K_\alpha$ by \eqref{EQ:solenoid-notLYC-lower-bound:orbit},
there are $t>s$ and $c\in\SIGMA^t$ starting with $b$ such that
$x_i\in K_c$ and $c\ne\alpha_{[0,t)}$. Thus $K_c<K_{\alpha[0,t)}$
by \eqref{EQ:solenoid-notLYC-lower-bound:Ka-less}, and
$$
  x_i\in K_c < \min K_{\alpha[0,t)}  \le \min K_\alpha = y_\alpha.
$$
Analogously, $x_{i+h}<y_\alpha$. Thus, by \eqref{EQ:solenoid-notLYC-lower-bound:diam},
$$
  \abs{x_{i+h}-x_i}
  \le \diam [y_b,y_\alpha) \le \lambda(K_b\setminus K_\alpha) < \eps.
$$
Hence \eqref{EQ:solenoid-notLYC-lower-bound:dist} is proved.

Since \eqref{EQ:solenoid-notLYC-lower-bound:dist} is true for every $h\in p_s\NNN$, we have
$\cccl_\infty(x,\eps)\ge (1/p_s) >0$. The proposition is proved.
\end{proof}

\subsection{Solenoidal $\omega$-limit sets without non-separable points --- the second part}
\label{SS:solenoidal-without-nonsepar2}
In this section we assume that $p_t=2^t$ for every $t\ge 1$; so $\SIGMA^t=\{0,1\}^t$
and $\SIGMAinf=\{0,1\}^{\NNN_0}$.
Fix $\eps>0$ and for every $t\ge 0$ put
\begin{equation*}
  \LLl_t(\eps) = \{a\in\SIGMA^t:\ \diam(K_a)\ge\eps\}
  \qquad
  \text{and}
  \qquad
  \LLl_\infty(\eps) = \{\alpha\in\SIGMAinf:\ \diam(K_\alpha)\ge\eps\};
\end{equation*}
denote the cardinalities of these (finite) sets by $\ell_t(\eps)$ and $\ell_\infty(\eps)$, respectively.

\begin{lemma}\label{L:solenoid-notLYC-upper-bound-longK} Fix any $\eps>0$.
Then
  \begin{enumerate}
    \item\label{IT:L:solenoid-notLYC-upper-bound-longK:1}
      $(\ell_t(\eps)\cdot 2^{-t})_t$ is non-increasing;
    \item\label{IT:L:solenoid-notLYC-upper-bound-longK:2}
      $(\lambda_t(\eps))_t$ is strictly decreasing,
      where $\lambda_t(\eps)=\sum_{a\in\LLl_t(\eps)} \diam(K_a)$;
    \item\label{IT:L:solenoid-notLYC-upper-bound-longK:3}
      $\ell_t(\eps)=\ell_\infty(\eps)$ for every sufficiently large $t$.
  \end{enumerate}
\end{lemma}
\begin{proof}
  If $a\not\in\LLl_t$, then trivially $[a]\cap \LLl_{t'}=\emptyset$
  for every $t'\ge t$. So
  $\ell_{t'}\le \ell_t\cdot 2^{t'-t}$  and
  \eqref{IT:L:solenoid-notLYC-upper-bound-longK:1} is proved.

  Since \eqref{IT:L:solenoid-notLYC-upper-bound-longK:2} is trivial, it suffices to prove
  \eqref{IT:L:solenoid-notLYC-upper-bound-longK:3}. For every $t$ define $\varphi_t:\LLl_{t+1}\to\LLl_t$
  by $\varphi_t(a) = a_{[0,t)}$; since $K_a\subseteq K_{\varphi_t(a)}$, $\varphi_t(a)\in \LLl_t$
  for every $a\in\LLl_{t+1}$. If $\varphi_t$ is not surjective, then
  $\lambda_{t+1}\le \lambda_t-\eps$. Combined with \eqref{IT:L:solenoid-notLYC-upper-bound-longK:2}
  this implies that $\varphi_t$ is surjective for all but finitely many $t$. Hence, for some $s$,
  \begin{equation}\label{EQ:L:solenoid-notLYC-upper-bound-longK:1}
    \ell_s\le \ell_{s+1}\le \ell_{s+2}\le\dots
  \end{equation}
  Since $\ell_t\le(1/\eps)$ for every $t$,
  by \eqref{EQ:L:solenoid-notLYC-upper-bound-longK:1} there is an integer $k$ such
  that $\ell_t=k$ for every sufficiently large $t$. Now trivially $\ell_\infty=k$.
\end{proof}

\medskip

For any $t\ge 0$ and $a\in\SIGMA^t$,
denote the words $a00,a01,a10,a11$
by $a_*,a_\circ,a^\circ, a^*$ in such a way that
\begin{equation*}
  K_{a_*} < K_{a_\circ} < K_{a^\circ} < K_{a^*}.
\end{equation*}
Since either $K_{a_*} \cup K_{a_\circ} \subseteq K_{a0}$ and
$K_{a^\circ} \cup K_{a^*} \subseteq K_{a1}$, or vice versa, we have
\begin{equation}\label{EQ:a*(s)}
  a_{*t} = a_{\circ t} \ne a^{\circ}_t =  a^{*}_ t.
\end{equation}
Put
\begin{equation*}
  \eps_t = \max_{a\in\SIGMA^t} \dist(K_{a_*},K_{a^*})
\end{equation*}
and note that $\eps_t > \diam(K_{a_\circ} \cup K_{a^\circ})$ for every $a\in\SIGMA^t$.

\begin{lemma}\label{L:solenoid-notLYC-upper-bound-epst}
  Let $f:I\to I$ be continuous. Let $x\in I$ be such that $\omega_f(x)$ is solenoidal
  and does not contain non-separable points. Then
  $$
   \lim_{t\to\infty}\eps_t = 0.
  $$
\end{lemma}
\begin{proof}
  By Lemma~\ref{L:solenoid-notLYC-upper-bound-longK}\eqref{IT:L:solenoid-notLYC-upper-bound-longK:3}
  there is $s$ such that, for every $a\in\SIGMA^*$ with $\abs{a}\ge s$,
  \begin{equation*}
    \diam(K_a)\ge\eps
    \qquad\text{if and only if}\qquad
    \text{there is } \alpha\in\SIGMAinf\cap[a] \text{ with } \diam(K_\alpha)\ge\eps.
  \end{equation*}
  Suppose that $\limsup_t \eps_t>0$; that is, there are $\bar\eps>0$, integers $s\le t_1<t_2<\dots$,
  and words $a^n\in\SIGMA^{t_n}$ ($n\in\NNN$), such that
  $\diam(K_{a^n})>\dist(K_{a^n_*},K_{a^{n*}})>\bar\eps$
  for every $n$. By going to a subsequence if necessary, we may assume that there is
  $\alpha\in\SIGMAinf$ such that $K_\alpha=\bigcap_n K_{a^n}$.
  By the choice of $\bar\eps$,
  $K_{a^n_*} < K_\alpha < K_{a^{n*}}$ for every sufficiently large $n$,
  and so both end points $y_\alpha,z_\alpha$ of $K_\alpha$ belong to $\omega_f(x)$.
  Hence, by Lemma~\ref{L:nonseparable}, $\omega_f(x)$ contains non-separable points.
\end{proof}

\begin{lemma}\label{L:solenoid-notLYC-upper-bound-words}
  Let $s\ge 0$, $t\ge s+2$, $a\in\SIGMA^s$, and $b\in\SIGMA^t\cap[a]$.
  Let $h$ be an odd multiple of $2^s$. Then for every $u,v\in\{0,1\}$
  there is a unique integer $i\in[0,3]$ such that
  either $b+i2^s\in[a0u]$ and $b+h+i2^s\in[a1v]$,
  or $b+i2^s\in[a1v]$ and $b+h+i2^s\in[a0u]$.
\end{lemma}
\begin{proof}
  We may assume that $s=0$, $t=2$, $h\in\{1,3\}$, and $b=00$.
  In such a case the statement of the lemma can be easily verified.
\end{proof}

\begin{lemma}\label{L:solenoid-notLYC-upper-bound-dist-infty1}
  Let $t\ge 2$ and $h$ be an odd multiple of $2^s$ for some $s\in[0, t-2]$.
  Then
  \begin{equation*}
    \dist_\infty(K_b,K_{b+h}) \ge \eps_s
    \qquad\text{for every } b\in\SIGMA^t.
  \end{equation*}
\end{lemma}
\begin{proof}
  Let $a\in\SIGMA^s$ be such that $\eps_s=\dist(K_{a_*},K_{a^*})$. Fix arbitrary $b\in\SIGMA^t$.
  Since $\dist_\infty(K_b,K_c)=\dist_\infty(K_{b+j},K_{c+j})$ for every $c$ and $j$, we may
  assume that $b\in[a]$.
  By \eqref{EQ:a*(s)} and Lemma~\ref{L:solenoid-notLYC-upper-bound-words},
  there is $i\in[0,3]$
  such that either $b+i2^s\in[a_*]$ and $b+h+i2^s\in[a^*]$, or vice versa.
  In both cases,
  \begin{equation*}
    \dist_\infty(K_b,K_{b+h})
    \ge \dist(K_{b+i2^s},K_{b+h+i2^s})
    \ge \dist(K_{a_*},K_{a^*})
    =\eps_s.
  \end{equation*}
\end{proof}

The next lemma follows immediately from Lemma~\ref{L:solenoid-notLYC-upper-bound-dist-infty1}.
\begin{lemma}\label{L:solenoid-notLYC-upper-bound-dist-infty2}
  Let $t\ge 2$ and $h$ be an integer which is not a multiple of $2^{t-1}$.
  Then, for every $b\in\SIGMA^t$,
  \begin{equation*}
    \dist_\infty(K_b,K_{b+h}) \ge \min\{\eps_0,\dots,\eps_{t-2}\}.
  \end{equation*}
\end{lemma}

\begin{lemma}\label{L:solenoid-notLYC-upper-bound-MAIN}
  Let $t\ge 2$ be such that $\eps_{t-2}=\min\{\eps_0,\dots,\eps_{t-2}\}$.
  Then
  \begin{equation*}
    \rdetuinfty (x,\eps_{t-2}) \le \frac 45 \,.
  \end{equation*}
\end{lemma}
\begin{proof}
  For abbreviation, write $\eps, N_m, N_m^\circ$ instead of $\eps_{t-2}, N_m(x,t,\eps_{t-2}), N_m^\circ(x,t,\eps_{t-2})$, respectively.
  For $m\in\{1,\infty\}$ and $h\in[0,2^t)$
  put
  $n_m(h)=\card\{b\in\SIGMA^t:\ \dist_m(K_b,K_{b+h})<\eps \}$
  and
  $n_m^\circ(h)=\card\{b\in\SIGMA^t:\ \diam_m(K_b, K_{b+h})\le\eps \}$;
  note that $N_m=\sum_h n_m(h)$ and $N_m^\circ=\sum_h n_m^\circ(h)$.
  By Lemma~\ref{L:solenoid-notLYC-upper-bound-dist-infty2},
  $n_\infty(h)=0$ for every $h\not\in\{0,2^{t-1}\}$; thus
  \begin{equation}\label{EQ:L:solenoid-notLYC-upper-bound-MAIN:1}
    N_\infty \le 2\cdot 2^t.
  \end{equation}
  On the other hand, if $h=2^{t-2}$ or $h=3\cdot 2^{t-2}$ then
  $n_1^\circ(h)\ge (1/4)\cdot2^t$;
  to see this, use \eqref{EQ:a*(s)}, Lemma~\ref{L:solenoid-notLYC-upper-bound-words}, and the fact
  that $\diam(K_{a_\circ}\cup K_{a^\circ}) < \dist(K_{a_*},K_{a^*})\le\eps$
  for every $a\in\SIGMA^{t-2}$.
  Hence
  \begin{equation}\label{EQ:L:solenoid-notLYC-upper-bound-MAIN:2}
    N_1^\circ - N_\infty \ge 2\cdot \frac {2^t}{4} = 2^{t-1}.
  \end{equation}
  Applying Lemma~\ref{L:corr-integ-and-intervals}\eqref{IT:L:corr-integ-and-intervals:5}
  and inequalities
  \eqref{EQ:L:solenoid-notLYC-upper-bound-MAIN:1},
  \eqref{EQ:L:solenoid-notLYC-upper-bound-MAIN:2} we conclude that
  \begin{equation*}
    \rdetuinfty (x,\eps)
    \le
    \frac{N_\infty}{N_1^\circ}
    \le
    \frac{N_\infty}{N_\infty + 2^{t-1}}
    \le
    \frac{1}{1+\frac 14}
    = \frac 45\,.
  \end{equation*}
\end{proof}

From
Lemmas~\ref{L:solenoid-notLYC-upper-bound-MAIN} and \ref{L:solenoid-notLYC-upper-bound-epst}
we immediately obtain the following proposition stating that, for solenoidal $\omega$-limit sets,
we cannot have recurrence determinism converging to one as $\eps\to 0$.
It is interesting that we even have an upper bound for limes inferior
depending neither on $f$ nor on $x$.

\begin{proposition}\label{P:solenoid-notLYC-upper-bound}
  Let $f:I\to I$ be continuous and $x\in I$ be such that $\omega_f(x)$ is solenoidal
  and does not contain non-separable points.
  Then
  \begin{equation*}
    \liminf_{\eps\to 0} \rdetuinfty (x,\eps) \le\frac 45 \,.
  \end{equation*}
\end{proposition}

\subsection{Proofs of Theorems~\ref{T:main} and \ref{T:omega}}\label{SS:proofs}

\begin{proof}[Proof of Theorem~\ref{T:omega}]
  It suffices to prove implications from left to right.
  Propositions~\ref{P:finite-omega} and \ref{P:solenoid-LYC} prove
  \eqref{IT:omega-finite} and \eqref{IT:omega-withLY} of Theorem~\ref{T:omega}.
  Assertion \eqref{IT:omega-nonLY} is shown in Propositions~\ref{P:solenoid-notLYC-lower-bound}
  and \ref{P:solenoid-notLYC-upper-bound}.
\end{proof}

\begin{proof}[Proof of Theorem~\ref{T:main}]
  We start by proving implications from left to right.
  The implication from \eqref{IT:T:main:1} follows immediately from Proposition~\ref{P:finite-omega}.
  If $f$ is not Li-Yorke chaotic then the $\omega$-limit sets of $f$ are either finite, or
  solenoidal without non-separable points  \cite[Theorem~2.2]{smital1986chaotic}.
  Thus Propositions~\ref{P:finite-omega} and \ref{P:solenoid-notLYC-lower-bound}
  show the implication from \eqref{IT:T:main:2}.

  Assume now that $f$ is Li-Yorke chaotic. If the topological entropy of $f$ is zero then,
  by \cite[Theorem~2.2]{smital1986chaotic}, there is $x\in I$ such that
  $\omega_f(x)$ contains non-separable points. Let $C$ be the (Cantor) set from
  Lemma~\ref{L:solenoidal-blokh}. Then $\omega_f(y)=C$ for every $y\in C$.
  Since $C$ contains non-separable points $z$ and $z'$ by Lemma~\ref{L:nonseparable-in-C},
  Proposition~\ref{P:solenoid-LYC} implies that
  $\rdetinfty(y,\eps)=0$ for every $y\in C$ and every $\eps\le \varrho(z,z')$.
  If the topological entropy of $f$ is positive, then \cite[Theorem~B]{spitalsky2016local}
  asserts that there is a Cantor set $C\subseteq I$ such that
  $(\cccu_m(x,\eps))_m$ decreases to zero exponentially fast
  (and hence $\rdetinfty(x,\eps)=\ccc_\infty(x,\eps)=0$)
  for every $x\in C$ and every
  sufficiently small $\eps>0$. Thus also the implication from \eqref{IT:T:main:3} is proved.

  Now we prove implications from right to left.
  Since the conditions on recurrence determinism
  from \eqref{IT:T:main:2} and \eqref{IT:T:main:3} are
  mutually exclusive, the equivalences from \eqref{IT:T:main:2} and \eqref{IT:T:main:3}
  are immediate. Assume now that $\rdetinfty(x,\eps)=1$ for every $x\in I$ and every
  sufficiently small $\eps>0$. By \eqref{IT:T:main:2}, $f$ is not Li-Yorke chaotic
  and thus have zero entropy.
  So every $\omega$-limit set of $f$ is finite by Theorem~\ref{T:omega}\eqref{IT:omega-finite}.
\end{proof}

\section{Proof of Theorem~\ref{T:example}}\label{S:construction}
Put $\SIGMA^t=\{0,1\}^t$ for $t\ge 1$, $\SIGMA^0=\{o\}$, and $\SIGMA^*=\bigcup_{t}\SIGMA^t$.
We say that a system $\KKk=\{K_a:\ a\in\SIGMA^*\}$
of disjoint non-degenerate closed subintervals $K_a=[y_a,z_a]$ of $I$
is \emph{admissible} if the following hold:
\begin{enumerate}
  \item\label{IT:K:1}
    $K_{o}=I$;
  \item\label{IT:K:2}
    $y_a=y_{a0}<y_{a1}<z_{a0}<z_{a1}=z_a$ for every $a\in\SIGMA^*$;
  \item\label{IT:K:3}
    $\nu_t = \max\{\diam K_a:\ a\in\SIGMA^t\}$ converges to $0$ as $t\to\infty$.
\end{enumerate}
Put $Q_t=\bigsqcup_{a\in\SIGMA^t} K_a$ and $Q=Q(\KKk)=\bigcap Q_t$; note that $Q$ is
a Cantor set.

\begin{lemma}\label{L:fK}
  Let $\KKk=\{K_a:\ a\in\SIGMA^*\}$ be admissible.
  Then there is a Li-Yorke non-chaotic continuous map
  $f:I\to I$ of type $2^\infty$ such that $Q(\KKk)$ is the only infinite
  $\omega$-limit set of $f$.
  Moreover, $f(K_a)=K_{a+1}$ for every $a\in\SIGMA^*$.
\end{lemma}
Every map $f$ with properties from the lemma will be called \emph{associated to $\KKk$}.
\begin{proof}
  Let $\tilde\KKk=\{\tilde{K}_a=[\tilde y_a,\tilde z_a]:\ a\in\SIGMA^*\}$ be
  the (admissible) system of intervals  defining the Cantor ternary set;
  that is, $\tilde y_{a1}=\tilde y_a + (1/3) (\tilde z_a - \tilde y_a)$
  and $\tilde z_{a0}=\tilde z_a - (1/3) (\tilde z_a - \tilde y_a)$
  for every $a\in\SIGMA^*$.
  Let $h$ be an increasing homeomorphism of $I$ such that
  $h(K_a)=\tilde K_a$ for every $a$.
  Define $f=h^{-1}\circ \tilde f \circ h$, where $\tilde f$ is Delahaye's map
  (see \cite{delahaye1980fonctions} or \cite[Example~5.56]{ruette2017chaos}).
  Since $f$ is conjugate to $\tilde f$, the dynamical properties of $f$
  are the same as those of $\tilde f$.
  Further, $f(K_a)=h^{-1}(\tilde f(\tilde K_a))=h^{-1}(\tilde K_{a+1})=K_{a+1}$ for every $a$.
\end{proof}

\begin{lemma}\label{L:fK-det1}
  Let $f:I\to I$ be associated to an admissible system $\KKk=\{K_a=[y_a,z_a]:\ a\in\SIGMA^*\}$,
  and $x\in I$ be such that $\omega_f(x)=Q(\KKk)$.
  Let $t\in\NNN$ and $\eps>0$ satisfy
  \begin{equation}\label{EQ:L:fK-det1}
    \diam(K_a)\le\eps \le \dist(K_a,K_b)
    \qquad\text{for every }  a\ne b
    \text{ from }\SIGMA^t.
  \end{equation}
  Then $\rdet_m(x,\eps)=1$ for every $m\in\NNN\cup\{\infty\}$.
\end{lemma}
\begin{proof}
  By \eqref{EQ:L:fK-det1}, $N_m^\circ(x,t,\eps)=N_m(x,t,\eps)
  =\card\{(a,a):\ a\in\SIGMA^t\}=2^t$
  for every $m$. Hence $\rdet_m(x,\eps)=1$
  by Lemma~\ref{L:corr-integ-and-intervals}\eqref{IT:L:corr-integ-and-intervals:5}.
\end{proof}

\begin{lemma}\label{L:fK-det0}
  Let $f:I\to I$ be associated to an admissible system $\KKk=\{K_a=[y_a,z_a]:\ a\in\SIGMA^*\}$,
  and $x\in I$ be such that $\omega_f(x)=Q(\KKk)$.
  Let $t>q\ge 1$ and $\eps>0$ satisfy
  \begin{equation}\label{EQ:L:fK-det0}
    \diam(K_a)\le\eps \le \dist(K_{1^{t-1}0},K_{1^t})
    \qquad\text{for every }  a\in\SIGMA^{t-q}\setminus\{1^{t-q}\}.
  \end{equation}
  Then $\rdetu_\infty(x,\eps)\le 2^{1-q}$.
\end{lemma}
\begin{proof}
  For any distinct $a,b\in\SIGMA^t$ we have $\dist_\infty(K_a,K_b)\ge\eps$.
  (To see this, take $i\in[0,2^t)$ such that $a+i=1^t$; since $b+i\ne a+i$, it holds that
  $\dist_\infty(K_a,K_b)\ge \dist(K_{1^t},K_{b+i})\ge\eps$ by \eqref{EQ:L:fK-det0}.)
  Thus $N_\infty(x,t,\eps)\le 2^t$.

  Take arbitrary $a\in\SIGMA^{t-q}\setminus\{1^{t-q}\}$ and any $b,c\in\SIGMA^t\cap [a]$.
  Then $\diam(K_b\cup K_c)\le \diam K_a \le\eps$ and hence
  $N_1^\circ (x,t,\eps)\ge (2^{t-q}-1)\cdot (2^q)^2$.
  By Lemma~\ref{L:corr-integ-and-intervals}\eqref{IT:L:corr-integ-and-intervals:5},
  \begin{equation*}
    \rdetu_\infty(x,\eps)
    \le \frac{2^t}{(2^{t-q}-1)\cdot (2^q)^2}
    \le \frac{1}{2^q\cdot\left(  1-2^{q-t} \right)}
    \le 2^{1-q}.
  \end{equation*}
\end{proof}

\begin{proof}[Proof of Theorem~\ref{T:example}]
  Put $t_1=1$ and, for every $n\ge 1$, $t_n'=t_n+(n+1)$ and $t_{n+1}=t_n'+1$.
  We are going to inductively define an admissible system $\KKk=\{K_a:\ a\in\SIGMA^*\}$;
  we will implicitly assume that all intervals $K_a$ are defined in such a way that
  \eqref{IT:K:2} from the definition of admissibility is satisfied.
  Put $K_o=[0,1]$, $K_0=[0,1/3]$, and $K_1=[2/3,1]$; then we have
  \eqref{EQ:L:fK-det1} with $t=t_1$ and $\eps=\eps_1=1/3$.

  Assume that, for some $n\in\NNN$ and $t=t_n$, we have already defined $K_a$ for every $a$ with
  $\abs{a}\le t$.
  Put $\eps_n=\max\{\diam K_a:\ a\in\SIGMA^t\}$,  $\delta=\diam K_{1^t}$,
  and take positive $\eps_n'<\delta/n$.
  For every $b\in\SIGMA^{t+1}$ define $K_b$
  in such a way that
  $\diam K_b <\eps_n'$ if $b\ne 1^{t+1}$,
  and $\diam K_b>n\eps_n'$ if $b= 1^{t+1}$.
  For $t+1<s\le t+n+1=t_n'$ and $b\in\SIGMA^s$ define $K_b$ arbitrarily
  requiring only that
  $\dist(K_{1^{s-1}0},K_{1^s})>\eps_n'$.
  In this way  we obtain that
  \eqref{EQ:L:fK-det0} is satisfied with $t=t'_n$, $q=n$, and $\eps=\eps'_n$.

  For $t=t'_{n+1}$ put $\eps_{n+1}=(1/3)\min\{\diam K_a:\ a\in\SIGMA^{t-1}\}$
  and define $K_b$ ($b\in\SIGMA^t$) in such a way that $\diam K_b=\eps_{n+1}$.
  Then \eqref{EQ:L:fK-det1} is satisfied with $t=t'_{n+1}$ and $\eps=\eps_{n+1}$.

  In this way we obtain an admissible system
  $\KKk=\{K_a=[y_a,z_a]:\ a\in\SIGMA^*\}$
  and sequences $(\eps_n)_n$, $(\eps'_n)_n$ decreasing to zero
  such that, for every $n$,
  \begin{itemize}
    \item
      \eqref{EQ:L:fK-det1} is satisfied with $t=t_n$ and $\eps=\eps_n$;
    \item
      \eqref{EQ:L:fK-det0} is satisfied with $t=t'_n$, $q=n$, and $\eps=\eps'_n$.
  \end{itemize}
  Let $f:I\to I$ be associated to $\KKk$.
  Put $C=Q(\KKk)$ and take any $x\in C$. 
  By Lemma~\ref{L:solenoidal-blokh}\eqref{IT:solenoidal-blokh}
  and the fact that $Q(\KKk)$ is totally disconnected, $C$
  is equal to the set $C$ from Lemma~\ref{L:solenoidal-blokh}
  and $\omega_f(x)=C$.
  Thus, by Lemmas~\ref{L:fK-det1} and \ref{L:fK-det0},
  $\rdetinfty(x,\eps_n)=1$ and $\rdetu_\infty(x,\eps'_n)\le 2^{1-n}$
  for every $n$. This proves Theorem~\ref{T:example}.
\end{proof}

\section*{Acknowledgements}
Substantive feedback from Michaela Mihokov\'a
is gratefully acknowledged.
This research is an outgrowth of the project ``SPAMIA'', M\v S SR-3709/2010-11,
supported by the Ministry of Education, Science, Research and Sport of the Slovak Republic,
under the heading of the state budget support for research and development.
The author also acknowledges support from VEGA~1/0786/15 and APVV-15-0439 grants.

\bibliography{rqa-intr}

\end{document}